\documentclass[preprint]{elsarticle}
\usepackage{amssymb, amsmath, amsthm, latexsym}
\usepackage{eucal}
\usepackage{ytableau, setspace}

\newtheorem{thm}{Theorem}[section]
\newtheorem{cor}[thm]{Corollary}
\newtheorem{lem}[thm]{Lemma}
\newtheorem{prop}[thm]{Proposition}
\newtheorem{proc}[thm]{Procedure}
\theoremstyle{definition}

\theoremstyle{remark}

\numberwithin{equation}{section}

\begin{document}

\title{Semistandard Tableaux for Demazure Characters \\ (Key Polynomials) and Their Atoms}

\author[RAP]{Robert A. Proctor}
\ead{rap@email.unc.edu}

\author[MJW]{Matthew J. Willis\fnref{fn1}}
\ead{mwillis1@conncoll.edu}

\address[RAP]{Dept. of Mathematics, University of North Carolina, Chapel Hill, NC 27599, USA}

\address[MJW]{Dept. of Mathematics, Hampden-Sydney College, Hampden-Sydney, VA 23943, USA}

\fntext[fn1]{Present Address: Dept. of Mathematics, Connecticut College, New London, CT 06320, USA}

\begin{abstract}
The Schur function indexed by a partition $\lambda$ with at most $n$ parts is the sum of the weight monomials for the Young tableaux of shape $\lambda$.  Let $\pi$ be an $n$-permutation.  We give two descriptions of the tableaux that contribute their monomials to the key polynomial indexed by $\pi$ and $\lambda$.  (These polynomials are the characters of the Demazure modules for $GL(n)$.)  The ``atom'' indexed by $\pi$ is the sum of weight monomials of the tableaux whose right keys are the ``key'' tableau for $\pi$.  Schur functions and key polynomials can be decomposed into sums of atoms.  We also describe the tableaux that contribute to an atom, the tableaux that have a left key equal to a given key, and the tableaux that have a left key bounded below by a given key.
\end{abstract}


\maketitle

\section{Introduction}

The core of this paper, Sections 4 - 9, is accessible to any mathematician.  After technical definitions are given in Section 2, the main definitions and details for the background material mentioned here appear in Section 3 (which is a second introductory section).

We think of ``Demazure'' (key) polynomials as being ``partial Schur functions'':  The Schur function $s_\lambda (x)$ is the sum of weight monomials for the semistandard tableaux of shape $\lambda$.  Via the notion of ``right key'', specification of an $n$-permutation $\pi$ determines a certain subset of those tableaux;  the sum of their weight monomials is the Demazure polynomial we denote $d_\lambda (\pi;x)$.  These polynomials give a filtration for $s_\lambda(x)$ indexed by the Bruhat order:  As $\pi$ increases, more of the monomials for $s_\lambda(x)$ are incorporated into $d_\lambda(\pi;x)$.

But in the big view it seems best to take the definition of Demazure polynomial to be the result of applying a sequence of divided difference operators corresponding to $\pi$ to a weight monomial specified by $\lambda$:  When studying flag varieties, Demazure developed \cite{De1} this formula to describe certain characters of a Borel subgroup of any semisimple Lie group. By 1990 Lascoux and Sch\"{u}tzenberger \cite{LS2} had developed a combinatorial description of these polynomials using the plactic algebra.  A central notion in their work was that of the right key of a given semistandard tableau.  They proved that $d_\lambda(\pi;x)$ arises when a tableau is allowed to contribute its monomial if and only if its right key is dominated by the tableau corresponding to $\pi$.  We quote this result in Theorem 3.1.

The second-listed author of this paper gave a simpler method for finding the right key of a tableau \cite{Wi2}.  Here we use his ``scanning'' method to present two new descriptions of these contributing ``Demazure tableaux'' which seem to be more direct and more accessible than those available.  Our descriptions of the possible tableau values for a given location depend upon the values of the tableau in the columns ``to the east'', or upon the values of the tableau in the locations ``to the southwest''.

Our main result Theorem 10.1 generalizes the following obvious proposition from semistandard tableaux to Demazure tableaux.  Let $\lambda$ be an $n$-partition.  Using the reversed (column, row) indexing of Section 2, for each $(l,k) \in \lambda$ set $\mathcal{Z}^{SW}_\lambda(T;l,k) := [T(l-1,k), T(l,k+1) -1]$ and $\mathcal{Z}^{SE}_\lambda(T;l,k) := [k,  min\{$ $T(l,k+1)-1, T(l+1,k) \} ]$.

\begin{prop}\label{prop:sstd}A tableau on the shape $\lambda$ is semistandard if and only if either of the following conditions is satisfied:

\noindent (a)  For all $(l,k) \in \lambda$ one has $T(l,k) \in \mathcal{Z}^{SE}_\lambda(T;l,k)$.

\noindent (b)  For all $(l,k) \in \lambda$ one has $T(l,k) \in \mathcal{Z}^{SW}_\lambda(T;l,k)$. \end{prop}

Lascoux also developed some other related notions and polynomials (mostly with Sch\"{u}tzenberger, but also more recently).  The sums of the monomials of the tableaux whose right keys are exactly a given key were also considered in \cite{LS2};  there they were also described with actions of operators.  Following Mason \cite{Mas}, we refer to these polynomials as ``atoms''.  Schur functions and key polynomials can be expressed as sums of atoms, where the sums run over certain permutations according to Bruhat orders.  The notion of the ``left key'' of a tableau was developed in \cite{LS1}.  In that paper Lascoux and Sch\"{u}tzenberger considered the tableaux whose left key is one specified key and whose right key is another specified key.  All of these considerations would lead us to initially consider eight tableaux description problems:  (Right or Left key of the tableaux) $\times$ (is Bounded by or is Equal to a given key) $\times$ (referring to values to the East or to the SouthWest).  In addition to the (R,B,E) and (R,B,SW) descriptions mentioned above, we also present (R, Eq, E), (L, Eq, SW), and (L, B, SW) descriptions.  These five (eight) descriptions can now (could then) be combined in various ways.  One can combine our (R, Eq, E) and (L, Eq, SW) descriptions to describe the tableaux of \cite{LS1} mentioned above.  To be nonzero, these polynomials should be indexed by intervals in Bruhat orders.  Our (R, B, SW) and (L, B, SW) descriptions can be combined in a more practical fashion to describe a generalization of Demazure polynomials that would be indexed by intervals in Bruhat orders.

Each of the five tableau theorems in Sections 5, 6, 8, and 9 is a generalization of or an analog of Proposition \ref{prop:sstd}.  There are two viewpoints for each of these theorems:  First, each result can be viewed as a ``theoretical'' characterization of the tableaux at hand.  Theorem 5.1 appears to currently be the most direct characterization of Demazure tableaux available.  This characterization is used in a sequel to this paper to prove the ``convex polytope'' result mentioned below.  Second, each result can be viewed as indicating a recursive procedure for constructing the tableaux at hand.  Section 7 presents an outline of the procedure corresponding to Theorem 6.1.

Here are some combinatorial descriptions of right keys and/or Demazure polynomials and/or atoms (that are specific to Type A):  Theorem 4.3 of \cite{LS2}, Theorems 1, 2, 5(1)(2)(3), and 6 of \cite{RS1}, Section 3 of \cite{RS2}, Appendix A.5 of \cite{Ful}, Theorems 4.1, 4.2, 4.7 and 4.10 of \cite{Le1}, Theorems 3 and 8 of \cite{Ava}, Theorem 1.2 and Corollary 5.1 of \cite{Mas}, Section 12.8 of \cite{LB}, Theorem 3.3.2 and Proposition 3.4.3 of \cite{Fer}, and Definition 4.3 of \cite{HLMvW}.  The Lakshmibai-Musili-Seshadri ``lifting'' criteria for Demazure tableaux is Definition 12.8.6 of \cite{LB}.  In Section 3 of \cite{RS2}, Reiner and Shimozono indicated how a minimal lifting of a given tableau could be found with a series of jeu de taquin ``two column swaps'', thereby computing its right key.  Our \cite{Wi2} instead justified this column swap method in terms of the ``frank'' tableau approach that is presented in \cite{Ful}, and then introduced the scanning method to more directly describe the result of the column swaps.  Some generalizations of the concepts of right and left keys to general Lie type are mentioned at the end of the appendix.

Demazure characters have been widely studied.   Why are atoms of interest?  For the study of symmetric polynomials such as the Macdonald polynomials, there has been a growing realization that it can be useful to broaden one's considerations to include closely related nonsymmetric polynomials.  Atoms have arisen as certain specializations of nonsymmetric Macdonald polynomials \cite{Ion} \cite{HHL} \cite{Mas} \cite{HLMvW} \cite{Fer}.  Haglund, Haiman, and Loehr referred to atoms as ``nonsymmetric Schur functions''.   Mason's combinatorial description of atoms here helped lead to our \cite{Wi2}.  Lascoux was recently studying Demazure, Schubert, Grothendieck, and nonsymmetric Macdonald polynomials from the viewpoint of divided difference operators.   When doing experiments in this context, one must express the empirical results in terms of the polynomials in some basis.  Here he found (personal communication) atoms to form a particularly useful basis for all polynomials that generalizes the basis of Schur functions for symmetric polynomials.  Combinatorial descriptions of atoms such as our Theorem 6.1 give finer information than do plactic or polynomial recursions.

Reiner and Shimozono's Theorem 25 of \cite{RS1} and Postnikov and Stanley's Theorem 14.1 of \cite{PS} related Demazure polynomials to flagged Schur functions for certain $\pi$.  Postnikov and Stanley then remarked that the sets of Gelfand patterns for the flagged Schur functions that arise in this way form convex polytopes.  In \cite{PW} we use Theorem 5.1 below to prove that the set of Demazure tableaux for $(\lambda, \pi)$ forms a convex polytope if and only if $\pi$ is ``$\lambda$-312 avoiding''.  A byproduct is a sharpening of the Theorem 25 of \cite{RS1} description of the relationship between Demazure polynomials and flagged Schur functions.  Also, this relationship is now stated at the tableau level.

Although the polynomials have provided the motivation, our results are set entirely within the finer context of tableaux.  The notions of right and left keys were reduced to two scanning descriptions in \cite{Wi2}.  So the core of this paper is concerned with comparing the tableau output of a scanning method to a given key tableau.  In Section 4 we present the scanning method for finding the right  key.  In Sections 5-9 we state and prove our descriptions of sets of tableaux that are constrained by given keys using our ``insider'' language of scanning tableaux.  In Section 10 we summarize our results for ``outsiders'' in terms of left and right keys and polynomials.  The optional appendix places Demazure polynomials into the representation theory context of \cite{Hum}.   All algebraic matters (including the actions of the symmetric group) are also deferred to the appendix, since these are not needed for our work with tableaux.

\section{Basic definitions and notation}

Let $p, q \in \mathbb{Z}$.  Set $[p,q] := \{ p, p+1, ... , q \}$.  Throughout the paper some $n \geq 1$ is fixed.  Set $[n] := [1, n]$ and $(n) = (1, 2, ... , n)$.

An \textit{$n$-partition} $\lambda$ is an $n$-tuple $(\lambda_1, \lambda_2, ... , \lambda_n )$ of integers with $\lambda_1 \geq \lambda_2 \geq ... \geq \lambda_n \geq 0$. Let $\Lambda_n^+$ denote the set of all $n$-partitions.  An \textit{$n$-permutation} $\pi$ is an $n$-tuple $(\pi_1, \pi_2, ... , \pi_n)$ with distinct entries from $[n]$.  Denote the \emph{set} of all such $n$-tuples by $S_n$.

Fix an $n$-permutation $\pi$.  For $1 \leq i \leq n-1$, define $s_i.\pi$ to be the $n$-tuple formed from $\pi$ by interchanging the values $i$ and $i+1$, wherever they may appear.  Given a sequence $i_1, ... , i_t$ for some $t \geq 1$, define $s_{i_t}...s_{i_1}.\pi := s_{i_t}.(...(s_{i_2}.(s_{i_1}.\pi))...)$.  If the \emph{composition} $s_{i_t}...s_{i_1}$ is such that $s_{i_t}...s_{i_1}.(n) = \pi$ with $t$ minimal, we say that it is a \emph{reduced composition} for $\pi$.  (Although $S_n$ does not need to be regarded as a group for the work performed in this paper, the appendix does present two action models for the symmetric group.)  Let $\tau_0$ denote the ``longest'' $n$-permutation $(n, n-1, ... , 2 , 1)$.

Let $x_1, ... , x_n$ be variables.  Let $P(x)$ be a polynomial in $x_1, ... , x_n$.  Re-use the symbols $s_i$ for $1 \leq i \leq n-1$ and define $s_i.P(x)$ to be the polynomial obtained by interchanging $x_i$ and $x_{i+1}$ in $P(x)$.  For $1 \leq i \leq n-1$ also define operators $\rho_i := (x_i - x_{i+1})^{-1} \circ (1-s_i) \circ x_i$ (multiply, swap, subtract, then divide) and $\bar{\rho_i} := \rho_i - 1$.  Within a monomial $x_1^{b_1} \cdots x_i^{b_i}x_{i+1}^{b_{i+1}} \cdots x_n^{b_n}$, if $b_i \geq b_{i+1}$ then the ``local symmetrizing'' operator $\rho_i$ replaces the $x_i^{b_i}x_{i+1}^{b_{i+1}}$ factors with the ``locally symmetric string'' that ``connects'' $x_i^{b_i}x_{i+1}^{b_{i+1}}$ to $x_i^{b_{i+1}}x_{i+1}^{b_{i}}$.  For example, suppose $n=4$.  Using unsubscripted variable names such as $x := x_2$ for readability, we have $\rho_2.w^3x^7y^4z^9 = [\frac{1-s_2}{x-y}x].w^3x^7y^4z^9 = w^3(x^7y^4 + x^6y^5 + x^5y^6 + x^4y^7)z^9$.  The operator $\bar{\rho_i}$ omits the first term.  Note that if $b_i = b_{i+1}$, then $\rho_i$ fixes $x_i^{b_1}...x_i^{b_i}x_{i+1}^{b_{i+1}}...x_n^{b_n}$ and $\bar{\rho_i}$ annihilates it.

Fix $\lambda \in \Lambda_n^+$.  The \textit{Young diagram} (or \textit{shape}) of  $\lambda$,  also denoted $\lambda$,  consists of  $\lambda_i$ left justified boxes in the  $i^{th}$ row for  $1 \leq i \leq n$.  Set $| \lambda | := \lambda_1 + \lambda_2 + ... + \lambda_n$.  To emphasize the importance of columns over rows,  the box in the  $j^{th}$ column and the  $i^{th}$ row is denoted $(j,i) \in \lambda$.  As in \cite{Wi2},  the column lengths of  $\lambda$ are denoted  $\zeta_1, \zeta_2, ... ,  \zeta_{\lambda_1}$.  A \textit{semistandard tableau}  $T$  \textit{of shape}  $\lambda$  is a filling of  $\lambda$  with elements of  $[n]$  such that its values  $T(j,i)$  satisfy  $T(j,i) \leq  T(j+1,i)$  and  $T(j,i) < T(j,i+1)$.  Use the value $k$ when $T(l-1,k)$ is referenced with $l=1$, use the value $n$ when $T(l+1,k)$ is referenced with $l = \lambda_k$, and use the value $n+1$ when $T(l,k+1)$ is referenced with $k = \zeta_l$.  Let  $\mathcal{T}_\lambda$  denote the set of all semistandard tableau of shape  $\lambda$.  For  $T, U \in \mathcal{T}_\lambda$,  we write  $T \leq U$  if  $T(j,i) \leq  U(j,i)$  for all  $(j,i) \in \lambda$;  here we say $T$ is \textit{dominated by} $U$.  For $T \in \mathcal{T}_\lambda$, let $m(T)$ denote the maximum of the values that appear at the bottoms of the columns of $T$ and let $max(T)$ denote the maximum of the values in $T$.  Clearly $max(T) = m(T)$.  For the empty tableau $(())$, define $m( \hspace{1mm} (()) \hspace{1mm} ) := 1$.  Given $T \in \mathcal{T}_\lambda$, its \textit{weight monomial} is $x^T := \prod_{i=1}^n x_i^{c_i}$, where $c_i$ is the number of values in $T$ equal to $i$.  A tableau $T \in \mathcal{T}_\lambda$ is a \textit{key} if the values in a column also appear in every column to the west of that column.  Given $\pi \in S_n$, the  $\lambda$-\textit{key} of  $\pi$  is the semistandard tableau  $Y_\lambda(\pi)$  of shape  $\lambda$ whose $j^{th}$ column is obtained by sorting  $\pi_1, \pi_2, ... , \pi_{\zeta_j}$  into ascending order and then entering these values from top to bottom.  The $(5,5,3,3,2,1)$-key of $(6,9,4,5,3,2,1,7,8)$ is the fourth tableau in Figure 1 below.  The key $Y_\lambda(\tau_0)$ is the unique maximal element of $\mathcal{T}_\lambda$.

To obtain an irredundant indexing of the Demazure polynomials, it is necessary to restrict the choice of $\pi$ relative to the $\lambda$ at hand:   Fix some $\lambda \in \Lambda_n^+$.  Let $q_1 < q_2 < ... < q_k$ for some $k \geq 0$ denote the distinct columns lengths of $\lambda$.  Set $Q_\lambda := \{ q_1, ... , q_k \} = \{ \zeta_1, ... , \zeta_{\lambda_1} \}$.  Set $q_0 := 0$ and $q_{k+1} := n$.  Note that for $1 \leq i \leq n-1$, we have $\lambda_i = \lambda_{i+1}$ if and only if $i \notin Q_\lambda$.  Let $S_n^\lambda$ denote the set of all $n$-permutations $\pi$ such that whenever $i, j \in [q_{r-1} + 1, q_r]$ with $i < j$ for some $1 \leq r \leq k+1$, then $\pi_i < \pi_j$.  Note that $| S_n^\lambda | = \frac{n!}{q_1!(q_2-q_1)! \cdots (n-q_k)!}$.  One has $Q_\lambda \supseteq [n-1]$ if and only if the parts of $\lambda$ are distinct.  Hence $S_n^\lambda = S_n$ if and only if $\lambda_1 > \lambda_2 > ... > \lambda_n \geq 0$.  The formation of the keys of shape $\lambda$ of the elements of $S_n^\lambda$ defines a bijection to the set of all keys of shape $\lambda$.  This formation process also defines a projection from $S_n$ to $S_n^\lambda$;   it is described in the appendix.  There it is noted that the dominance ordering of the keys for $S_n^\lambda$ describes the Bruhat ordering on the $W^\lambda$ ``quotient'' manifestation of $S_n^\lambda$.  We borrow the semidirect product symbol to denote the subset of $\Lambda_n^+ \times S_n$ consisting of all $(\lambda, \pi)$ such that $\pi \in S_n^\lambda$:  This restriction of the set product is denoted with the visually suggestive $\Lambda_n^+ \rtimes S_n^\lambda$ (rather than with $\Lambda_n^+ \times \hspace{-.5mm} |_\lambda \hspace{1mm} S_n$).

\section{Cited results;  Demazure polynomial and tableau definitions}

Fix $(\lambda, \pi) \in \Lambda_n^+ \times S_n$.  Let $s_{i_t}...s_{i_2}s_{i_1}$ be reduced for $\pi$.  The operators $s_i$ satisfy $s_i s_{i+1} s_i = s_{i+1} s_i s_{i+1}$ for $1 \leq i \leq n-1$ and $s_is_j = s_js_i$ for $i,j \in [n]$ with $|i-j|>1$, and these relations can be used to relate any two reduced compositions for $\pi$.  We take the Demazure character formula as our definition of the \textit{Demazure polynomial}; that is $d_\lambda(\pi;x) := \rho_{i_t}...\rho_{i_2}\rho_{i_1}.x_1^{\lambda_1}x_2^{\lambda_2} \cdots x_n^{\lambda_n}$.  Since the analogous relations $\rho_i \rho_{i+1} \rho_i = \rho_{i+1} \rho_i \rho_{i+1}$ and $\rho_i\rho_j = \rho_j\rho_i$ also hold, these polynomials are well-defined functions of $\pi$ (and $\lambda$).

For a tableau $T \in \mathcal{T}_\lambda$, the \textit{right key} $R(T)$ is a certain key in $\mathcal{T}_\lambda$ that can be defined using a jeu de taquin process, as in Appendix A.5 of \cite{Ful}.  The following result of \cite{LS2} appeared as Theorem 1 in \cite{RS1}:

\begin{thm}The Demazure polynomial $d_\lambda(\pi;x)$ is the sum of the weight monomials $x^T$ for $T \in \mathcal{T}_\lambda$ such that $R(T) \leq Y_\lambda(\pi)$.\end{thm}

\noindent Hence we say that $T \in \mathcal{T}_\lambda$ is a \textit{Demazure tableau for $\pi$} if $R(T) \leq Y_\lambda(\pi)$.  Let $\mathcal{D}_\lambda(\pi)$ denote the set of such tableaux.  Reiner and Shimozono referred to the polynomials $d_\lambda(\pi;x)$ as the ``key polynomials'' $\kappa_\alpha (x)$ for ``compositions'' $\alpha \in \mathbb{N}^n$.  Our definition of the $d_\lambda(\pi;x)$ largely follows their definition of the $\kappa_\alpha (x)$.  Their Theorem 1 can be obtained from the second identity stated in Theorem 4.3 of \cite{LS2} by extracting the terms of degree $| \lambda |$ and projecting the resulting identity to polynomials in $n$ commuting variables.  With respect to $s_\lambda(x) = \sum_{T \in \mathcal{T}_\lambda} x^T$, one can view a Demazure polynomial as a ``partial Schur function''.  Since $Y_\lambda(\tau_0)$ is the unique maximal element of $\mathcal{T}_\lambda$, we have $R(T) \leq Y_\lambda(\tau_0)$ for all $T \in \mathcal{T}_\lambda$.  Thus $s_\lambda(x)$ is the Demazure polynomial $d_\lambda(\tau_0;x)$.

We define the \textit{atom} $c_\lambda(\pi;x) := \bar{\rho}_{i_t} ... \bar{\rho}_{i_2} \bar{\rho}_{i_1}.x_1^{\lambda_1}x_2^{\lambda_2} \cdots x_n^{\lambda_n}$.  This notion is well-defined by similar reasoning.  The following result is a consequence of Theorem 3.8 of \cite{LS2}:

\begin{thm}If $\pi \in S_n^\lambda$, the atom $c_\lambda(\pi;x)$ is the sum of the weight monomials $x^T$ for $T \in \mathcal{T}_\lambda$ such that $R(T) = Y_\lambda(\pi)$.\end{thm}

\noindent In the appendix it is noted that $c_\lambda(\pi;x) \neq 0$ if and only if $\pi \in S_n^\lambda$.

Fix $(\lambda, \pi) \in \Lambda_n^+ \rtimes S_n^\lambda$.  We say that $T \in \mathcal{T}_\lambda$ is an \textit{exact Demazure tableau at $\pi$} if $R(T) = Y_\lambda(\pi)$.  Let $\mathcal{C}_\lambda(\pi)$ denote the set of such tableaux. In \cite{LS2}, the element of the free algebra that projected to $c_\lambda(\pi;x)$ was called a ``standard basis''.  Our development here reverses the roles of ``definition'' and ``theorem'' for standard bases played by Definition 3.7 and Theorem 3.8 of \cite{LS2}.  The set $\mathcal{D}_\lambda(\pi)$ is the union of the sets $\mathcal{C}_\lambda(\pi^\prime)$ over $\pi^\prime \in S_n^\lambda$ such that $Y_\lambda(\pi^\prime) \leq Y_\lambda(\pi)$.  The analogous polynomial statement is $d_\lambda(\pi;x) = \sum c_\lambda(\pi^\prime; x)$.  In particular, one has $s_\lambda(x) = \sum c_\lambda(\pi^\prime;x)$, where the sum is over all $\pi^\prime \in S_n^\lambda$.

Let $\pi$ and $\pi^\prime$ be any two $n$-permutations.  As in the appendix, let $w$ and $w^\prime$ be the corresponding Weyl group elements.  The structures that provide the environment in which the entities of this paper are defined, the Demazure modules $D_\lambda(w)$, can be created for unrestricted $\pi$.  However, $D_\lambda(w) = D_\lambda(w^\prime)$ if and only if $w.\lambda = w^\prime.\lambda$, and the stabilizer of $\lambda$ is non-trivial if and only if $\lambda$ does not have distinct parts.  One also has $d_\lambda(\pi;x) = d_\lambda(\pi^\prime;x)$ if and only if $w.\lambda = w^\prime.\lambda$.  But the situation for atoms is different.  One of the referees for this paper caught the following error:  The restriction $\pi \in S_n^\lambda$ had to be added to Theorem 3.2 and Corollaries 10.4 and 10.6 since $c_\lambda(\pi;x) = 0$ when $\pi \notin S_n^\lambda$.  Nonetheless, none of our six theorems in Section 5 - 9 (which pertain to tableaux) really need the restriction $\pi \in S_n^\lambda$ for their statements or for their proofs.  There is a natural bijection from $\Lambda_n^+ \rtimes S_n^\lambda$ to $\mathbb{N}^n$ (where $\mathbb{N} := \{ 0, 1, 2, ... \} )$.  In fact, the key polynomials of [RS1] are indexed by elements of the latter set and it could be argued that $\mathbb{N}^n$ is the more natural indexing set for Demazure polynomials.  But, for a fixed shape $\lambda$, the goal of this paper is to identify the relevant semistandard tableaux of that shape.  Moreover, Demazure polynomials may at times be defined elsewhere for general $\pi \in S_n$.  Thus we will emphasize $\lambda$ in our notation, and in this edition we impose the requirement of $(\lambda, \pi) \in \Lambda_n^+ \rtimes S_n^\lambda$ when atoms or their related structures are present.

According to Corollary 7 of [RS1], as $(\lambda, \pi)$ runs through $\Lambda_n^+ \rtimes S_n^\lambda$ the set $\{ d_\lambda(\pi;x) \}$ (and hence the set $\{ c_\lambda(\pi;x) \}$) forms an integral basis for $\mathbb{Z}[x_1,...,x_n]$.  If $\lambda = (1, 0, ... , 0)$, the atoms are $x_1, x_2, ... , x_n$.

The first paper in this series gave a ``scanning method'' for computing the right key $R(T)$ of a tableau $T$.  This method is described in the next section;  its output is denoted $S(T)$.  Here is Theorem 4.5 of \cite{Wi2}:

\begin{thm}Let $T \in \mathcal{T}_\lambda$.  Then $R(T) = S(T)$.\end{thm}

\noindent This view of $R(T)$ made the following known result readily apparent:

\begin{cor}Let $T \in \mathcal{T}_\lambda$.  Then $T \leq R(T) = S(T)$.\end{cor}

Let $T \in \mathcal{T}_\lambda$.  The \textit{left key} $L(T)$ of $T$ is a certain key in $\mathcal{T}_\lambda$ that is defined and may be found in manners analogous to those for the right key \cite{Ful} \cite{Wi2}.  The scanning description in \cite{Wi2} easily confirms that $L(T) \leq T$.

\section{The scanning tableau S(T)}

Fix $\lambda \in \Lambda_n^+$ and a tableau $T \in \mathcal{T}_\lambda$.  Here we recall the scanning method of \cite{Wi2} for constructing the ``scanning tableau'' $S(T)$ of $T$.  Given a sequence of integers $(x_1, x_2, ...)$, define its \emph{earliest weakly increasing subsequence (EWIS)} to be $(x_{i_1}, x_{i_2}, ...)$, where $i_1 = 1$ and when $j > 1$ then $i_j$ is minimal such that $x_{i_{j-1}} \leq x_{i_j}$.

To follow the specification of this method, let $T$ be the first tableau in Figure 1.  Its ``scanning paths'' that originate in its first column will be indicated on the second tableau with the superscripts $a, b, ..., f$, and $S(T)$ will be the third tableau.

\begin{figure}[h!]
$$
\ytableausetup{boxsize=1.3em}
\begin{ytableau}
1 & 2 & 3 & 3 & 5 \\
2 & 3 & 4 & 8 & 9 \\
3 & 4 & 5 \\
4 & 6 & 8\\
5 & 8\\
7 \\
\end{ytableau}\hspace{2.5mm}
\begin{ytableau}
1^f & 2^e & 3^e & 3^e & 5^c \\
2^e & 3^d & 4^d & 8^a & 9^a \\
3^d & 4^c & 5^c \\
4^c & 6^b & 8^a\\
5^b & 8^a\\
7^a \\
\end{ytableau}\hspace{2.5mm}
\begin{ytableau}
1 & 3 & 3 & 5 & 5 \\
3 & 4 & 4 & 9 & 9 \\
4 & 5 & 5 \\
5 & 6 & 9\\
6 & 9\\
9 \\
\end{ytableau}\hspace{2.5mm}
\begin{ytableau}
2 & 3 & 4 & 6 & 6\\
3 & 4 & 5 & 9 & 9 \\
4 & 5 & 6 \\
5 & 6 & 9\\
6 & 9\\
9 \\
\end{ytableau}\hspace{2.5mm}
\begin{ytableau}
3 & 3 & 5 \\
4 \\
5 \\
\end{ytableau}
$$\caption{Tableaux for Section 4 and 5 examples.}
\end{figure}

Let $1 \leq l \leq \lambda_1$.  Create $T^{(l, \zeta_l)}$ from $T$ by removing the first $l-1$ columns from $T$ and $\lambda$, but retain the column indexing.  We compute the values in the $l^{th}$ column of $S(T)$ from $(l, \zeta_l)$ upwards:  Consider the column bottom values $T^{(l,\zeta_l)}(h, \zeta_h)$ for $l \leq h \leq \lambda_1$ as a sequence, and find its EWIS.  The sequence of locations that contain the values of this EWIS is the \emph{scanning path} for this location; it is denoted $P(T;l,\zeta_l)$.  The first member of $P(T;l,\zeta_l)$ is the location $(l,\zeta_l)$.  Begin to create $S(T)$ by defining the value $S(T;l,\zeta_l)$ to be the last value in this EWIS.  Next remove the boxes in $P(T; l, \zeta_l)$ from $\lambda$ and their values from $T$ to form what can be seen to be a smaller shape and a \textit{remnant} tableau $T^{(l;\zeta_l-1)}$.  Since $T^{(l;\zeta_l-1)}$ is semistandard, we may apply $S(\cdot)$ to it.    As $k$ decrements from $\zeta_l - 1$ to 1, continue to perform this process using the bottom values in the $l^{th}$ through $\lambda_1^{th}$ columns of the diminishing $T^{(l;k)}$ to produce the other $\zeta_l - 1$ scanning paths that originate in the $l^{th}$ column.  For such $k$, the path constructed with the selected column bottoms of $T^{(l;k)}$ is denoted $P(T;l,k)$, and $S(T;l,k)$ is defined to be the value in its final location.  Note that $S(T;l,k)$ is the largest of the column bottom values in $T^{(l,k)}$, i.e. the largest value in $T^{(l,k)}$.  Apply this process to all of the columns of $T$ to obtain the scanning value $S(T; l,k)$ for every $(l,k) \in \lambda$.  Define $U^{(l,k)}$ to be the tableau produced by removing the leftmost remaining column from $T^{(l,k)}$ and $\lambda$.  To summarize, with the second equality giving the form used in Sections 5 - 7:

\begin{lem}\label{lem:max}Let $\lambda \in \Lambda_n^+$.  Let $T \in \mathcal{T}_\lambda$ and $(l,k) \in \lambda$.  Then $S(T;l,k) = max(T^{(l,k)}) = max\{ (T(l,k), m(U^{(l,k)}) \}$ . \end{lem}

Recall that the fourth tableau in Figure 1 is $Y_\lambda(\pi)$ for $\lambda = (5,5,3,3,2,1)$ and $\pi = (6,9,4,5,3,2,1,7,8)$.  For the first tableau $T$ we have $S(T) \leq Y_\lambda(\pi)$.  Since $S(T) = R(T)$, we have $T \in \mathcal{D}_\lambda(\pi)$.

\section{Right key dominated by a given key (from the east)}

Fix $(\lambda, \pi) \in \Lambda_n^+ \times S_n$ and form $Y_\lambda(\pi) =: Y$.  Let $T \in \mathcal{T}_\lambda$.  Here we give necessary and sufficient conditions on the values in $T$ so that its scanning tableau $S(T)$ is dominated by $Y_\lambda(\pi)$.

Fix some $(l,k) \in \lambda$.  We define a set $A_\lambda(T,\pi;l,k)$ that contains the ``allowable'' values for $T$ at the location $(l,k)$.  Form $U^{(l,k)} =: U$ as in Section 4.  If $m(U) > Y(l,k)$, define $A_\lambda(T,\pi;l,k) := \emptyset$.  If $m(U) \leq Y(l,k)$, define $A_\lambda(T,\pi;l,k) := [ k , min \{ Y(l,k), T(l,k+1) -1, T(l+1, k) \} ] $.

Let $T$ be the first tableau in Figure 1 and let $(l,k) = (2,4)$.  Here $U^{(2,4)}$ is the fifth tableau in Figure 1:  It is obtained by removing the first column of $T$ and the values $(8,8,8,9)$ in $P(T;2,5)$ from $T$ to produce $T^{(2,4)}$, and then removing the leftmost column.  Note that $m(U) = 5 \leq 6 = Y(2,4)$.  Thus $A_\lambda(T,\pi;2,4) = [4, min\{ Y(2,4), T(2,5) - 1, T(3,4) ] = [ 4  , min\{6,8-1,8\} ] = [4,6]$.

\begin{thm}Given $(\lambda, \pi) \in \Lambda_n^+ \times S_n$, let $T \in \mathcal{T}_\lambda$.  Then $S(T) \leq Y_\lambda(\pi)$ if and only if $T(l,k) \in A_\lambda(T,\pi;l,k)$ for all $(l,k) \in \lambda$.\end{thm}

This result can be used in a procedure similar to Procedure 7.1 to construct Demazure tableaux for $(\lambda, \pi)$:  Suppose that the columns to the east and the boxes to the south of the location at hand in its column have been filled in with ``good-so-far'' values.  Find and remove the scanning paths originating from those boxes to the south.  If any of the column bottoms to the east in the remnant tableaux exceed the value of the $\lambda$-key for $\pi$ in the location at hand, then give up.  Otherwise one is free to choose any of the usual values from Proposition 1.1(a) for the location at hand, provided that one does not exceed the given key value there.

\begin{proof}Write $S(T) =: S$.  Let $(l,k) \in \lambda$.  Since $T$ is semistandard we have $k \leq T(l,k) \leq min \{ T(l,k+1) - 1, T(l+1,k) \}$.  By Corollary 3.4 we have $T(l,k) \leq S(l,k).$  Lemma \ref{lem:max} says $S(l,k) = max \{ T(l,k), m(U) \}$.

First suppose that $S \leq Y$ for $T$.  So $T(l,k) \leq Y(l,k)$.  And since $max \{ T(l,k),$ $ m(U) \}\leq Y(l,k)$, the set $A_\lambda(T,\pi;l,k)$ is non-empty.  Thus $T(l,k) \in A_\lambda(T,\pi;l,k)$.

Next suppose that $T(j,i) \in A_\lambda(T,\pi;j,i)$ for all $(j,i) \in \lambda$.  Since $A_\lambda(T,\pi;l,k)$ is non-empty, we have $m(U) \leq Y(l,k)$. Also we have $T(l,k) \leq Y(l,k)$.  Hence $S(l,k) \leq Y(l,k)$.\end{proof}

\section{Right key equal to a given key}

For this section and Section 7, fix $(\lambda, \pi) \in \Lambda_n^+ \rtimes S_n^\lambda$, and set $Y := Y_\lambda(\pi)$.  Let $T \in \mathcal{T}_\lambda$.  Here we give necessary and sufficient conditions on the values in $T$ so that its scanning tableau $S(T)$ is equal to $Y_\lambda(\pi)$.

Fix some $(l,k) \in \lambda$.  We now define a set $C_\lambda(T,\pi;l,k)$ that contains the allowable values for $T$ at the location $(l,k)$:  If $l = \lambda_1$, then set $C_\lambda(T,\pi;l,k) := \{ Y(l,k) \} $ for $1 \leq k \leq \zeta_{\lambda_1}$.  Suppose $\lambda_1 > l \geq 1$.  Form $U$ from $T^{(l;k)}$ as in Section 5.  If $m(U) > Y(l,k)$, set $C_\lambda(T,\pi;l,k) := \emptyset$.  If $m(U) = Y(l,k)$, set $C_\lambda(T,\pi;l,k) := [ k, min \{ Y(l,k), T(l,k+1)-1, T(l+1,k) \} ]$.  If $m(U) < Y(l,k)$, set $C_\lambda(T,\pi;l,k) := \{ Y(l,k) \} \bigcap [ k, min \{ T(l,k+1) - 1, T(l+1,k) \} ]$.  An example of a set $C_\lambda(T,\pi;l,k)$ appears after the statement of Procedure 7.1.

\begin{thm}Given $(\lambda, \pi) \in \Lambda_n^+ \rtimes S_n^\lambda$, let $T \in \mathcal{T}_\lambda$.  Then $S(T) = Y_\lambda(\pi)$ if and only if $T(l,k) \in C_\lambda(T,\pi;l,k)$ for all $(l,k) \in \lambda$.\end{thm}

\begin{proof}The beginning of this proof is the same as the first paragraph of the proof of Theorem 5.1.

First suppose that $S=Y$ for $T$.  So $T(l,k) \leq Y(l,k)$.  Since $T$ is semistandard we have $k \leq T(l,k) \leq min \{ T(l,k+1) - 1, T(l+1,k) \}$.  Here we have $max \{ T(l,k), m(U) \} = Y(l,k)$, and hence $m(U) \leq Y(l,k)$.  If $m(U) < Y(l,k)$, then we must have $T(l,k) = Y(l,k)$ in order to have $S(l,k) = Y(l,k)$.  So here $T(l,k) \in C_\lambda(T,\pi;l,k)$.  If $m(U) = Y(l,k)$, one also has $T(l,k) \in C_\lambda(T,\pi;l,k)$.

Next suppose that $T(j,i) \in C_\lambda(T,\pi;j,i)$ for all $(j,i) \in \lambda$.  Since $C_\lambda(T,\pi;l,k)$ is non-empty, we have $m(U) = Y(l,k)$ or $m(U) < Y(l,k)$.  In the former case, having $T(l,k) \leq Y(l,k)$ implies that $S(l,k) = Y(l,k)$.  In the latter case, the definition of $C_\lambda(T,\pi;j,i)$ implies $T(l,k) = Y(l,k)$.  So $Y(l,k) \leq S(l,k) = max \{ T(l,k), m(U) \}$.  Now $Y(l,k) < max \{ T(l,k), m(U) \}$ would imply $Y(l,k) < m(U)$, which is impossible here.  Hence $Y(l,k) = S(l,k)$.\end{proof}

\section{Generation of tableaux for an atom}

We continue to work in the context established in Section 6.  Here we present a recursive implementation of Theorem 6.1:  It generates all tableaux $T$ of shape $\lambda$ that have their scanning tableau $S(T)$ equal to the $\lambda$-key $Y_\lambda(\pi)$.  This procedure constructs each of the desired tableaux from east to west.  (The generation procedure on p. 281 of \cite{Le1} produces all of $\mathcal{D}_\lambda(\pi)$.)

For $\lambda_1 \geq l \geq 1$, denote the partition with column lengths $\zeta_l, \zeta_{l+1}, ... , \zeta_{\lambda_1}$ by $\lambda^{[l]}$.  In the description below, lower portions of the pending new column are denoted by $L$ and the empty pending column is denoted $()$.  Each pending new column $L$ (that will be extended upward) needs to be accompanied by an updated (shrinking upwards) partial tableau $U$.  The sets of potential new values are denoted by $C$.  The columns of the growing tableau $T$ and of the shrinking tableau $U$ are indexed from the right by $\lambda_1, \lambda_1 - 1, \lambda_1 - 2, ...$.

\begin{proc}Input $\lambda \in \Lambda_n^+$ and $\pi \in S_n^\lambda$.  Let $\mathcal{V}^{[\lambda_1]}$ be the set consisting of the one tableau $T$ of shape $\lambda^{[\lambda_1]}$ that is formed by taking the last column of $Y_\lambda(\pi) =: Y$.  As $l$ decrements from $\lambda_1 - 1$ to $1$, successively form sets $\mathcal{V}^{[l]}$ of tableaux of shapes $\lambda^{[l]}$ as follows:

\noindent For each $T \in \mathcal{V}^{[l+1]}$, do:

\noindent Let $\mathcal{F}_{\zeta_l+1}$ be the set consisting of the one ordered pair $( (), T )$.

\noindent As $k$ decrements from $\zeta_l$ to $1$, successively build up sets $\mathcal{F}_k$ of ordered pairs as follows:

\noindent For each $(L,U) \in \mathcal{F}_{k+1}$, do:

\noindent When $k < \zeta_l$, let $t$ be the first (northernmost) value in $L$; when $k = \zeta_l$, let $t$ be $n+1$.

\noindent If $m(U) > Y(l,k)$, set $C := \emptyset$.

\noindent If $m(U) = Y(l,k)$, set $C := [ k , min \{ Y(l,k), t-1, T(l+1,k) \} ]$.

\noindent If $m(U) < Y(l,k)$, set $C := \{ Y(l,k) \} \bigcap [ k , min \{ t-1, T(l+1,k) \} ]$.

\noindent If $C$ is empty, then discard $(L,U)$.

\noindent Let $\mathcal{F}(L,U)$ be the set of all ordered pairs $(L^\prime, U^\prime)$ that can be formed by prepending an element $z$ of $C$ to $L$ and then forming $U^\prime$ by deleting from $U$ the values and the boxes that lie in the scanning path in $U$ that originates from the value $z$ at the location $(l,k)$.  Let $\mathcal{F}_k$ be the union of the $\mathcal{F}(L,U)$ as $(L,U)$ runs through $\mathcal{F}_{k+1}$.  If $\mathcal{F}_k$ is empty, then discard $T$.  (When $k = 1$, each $U^\prime$ will be the null tableau (()) on the empty shape.)

\noindent After $k=1$, form the elements of $\mathcal{V}^{[l]}$ descended from this $T$ by prepending each column $L$ that appears in a pair $(L, (()) )$ in $\mathcal{F}_1$ to the tableau $T$.  Continue to the next $T \in \mathcal{V}^{[l+1]}$.

\noindent After $l = 1$, output the set of semistandard tableaux $\mathcal{V}^{[1]}$.\end{proc}

Suppose $\lambda = (4, 4, 3, 3, 2, 1, 1)$ and $\pi = (6, 8, 3, 7, 4, 1, 9, 2, 5)$.  Then $Y_\lambda(\pi)$ is the second tableau in Figure 2.  Let $T$ be the first tableau in Figure 2 and let $(l,k) = (2,3)$.  Here $U$ is the third tableau in Figure 2.  Hence $m(U) = 6 = Y(2,3)$, and so $C_\lambda(T,\pi;2,3) = [ 3, min \{ 6, 6-1, 7 \} ] = [3, 5] = \{ 3, 4, 5 \}$.

\begin{figure}[h!]
$$
\begin{ytableau}
1 & 1 & 3 & 6 \\
2 & 3 & 4 & 8 \\
4 & 5 & 7 \\
5 & 6 & 8\\
6 & 7\\
7 \\
9 \\
\end{ytableau}\hspace{8mm}
\begin{ytableau}
1 & 3 & 3 & 6 \\
3 & 4 & 6 & 8 \\
4 & 6 & 7 \\
6 & 7 & 8\\
7 & 8\\
8 \\
9 \\
\end{ytableau}\hspace{8mm}
\begin{ytableau}
3 & 6 \\
4 \end{ytableau}\hspace{8mm}
\begin{ytableau}
\color{white}1 & \color{white}1 & 3 & 6 \\
\color{white}2 & \color{white}3 & 4 & 8 \\
\color{white}4 & \color{white}5 & 7 \\
\color{white}5 & 6 & 8\\
\color{white}6 & 7\\
\color{white}7 \\
\color{white}9 \\
\end{ytableau}
$$\caption{Tableaux for Section 6 and 7 examples.}
\end{figure}

Within Procedure 7.1, again let $(l,k) = (2,3)$ and now suppose that the values in the fourth (partial) tableau in Figure 2 have been chosen so far.  Note that these values come from $T$, and so $C$ is the set $C_\lambda(T,\pi;2,3)$ above.  In Figure 3 the respective cases for the potential values 3, 4, and 5 from $C$ are indexed with the subscripts $a, b, c$ within $\mathcal{F}(L,U)$.  For $(L^\prime_a, U^\prime_a)$, we have $m(U^\prime_a) = 3 < 4 = Y(2,2)$.  Thus $C = \{ 4 \} \bigcap $ $[ 2, min \{2, 4 \} ] = \emptyset$, and so $(L^\prime_a, U^\prime_a)$ should be discarded. The same applies to $(L^\prime_b, U^\prime_b)$.  However, for $(L^\prime_c, U^\prime_c)$, we have $C = [ 2 , min \{ 4, 4, 4 \} ] = [2,4]$.  Hence this process can be continued.  In fact, this partial tableau can be filled entirely to produce a tableau that satisfies the requirements of Theorem 6.1.  The tableau $T$ is one such tableau.

\begin{figure}[h!]
$$\mathcal{F}(L,U) =  \{ (L^\prime_a, U^\prime_a) = ( \hspace{.8mm} \begin{ytableau} 3 \\ 6 \\ 7 \end{ytableau}, \begin{ytableau} 3 \end{ytableau} \hspace{.8mm} ) ,
(L^\prime_b, U^\prime_b) = ( \hspace{.8mm} \begin{ytableau} 4 \\ 6 \\ 7 \end{ytableau},  \begin{ytableau} 3 \end{ytableau} \hspace{.8mm} ) ,
(L^\prime_c, U^\prime_c) = ( \hspace{.8mm} \begin{ytableau} 5 \\ 6 \\ 7 \end{ytableau}, \begin{ytableau} 3 \\ 4 \end{ytableau} \hspace{.8mm} )  \} .$$
\caption{Procedure 7.1}
\end{figure}

Once the proof of Theorem 6.1 is understood, it should be clear that Procedure 7.1 does indeed generate all of the exact Demazure tableaux at $\pi$:

\begin{thm}Let $(\lambda, \pi) \in \Lambda_n^+ \rtimes S_n^\lambda$.  Then $\mathcal{V}^{[1]} = \{ T \in \mathcal{T}_\lambda \hspace{2mm} | \hspace{2mm} S(T) = Y_\lambda(\pi) \}$.\end{thm}

\section{Right key dominated by a given key (from the southwest)}

As in Section 5, fix $(\lambda, \pi) \in \Lambda_n^+ \times S_n$.  Here we show the scanning tableau of a given $T$ is dominated by the $\lambda$-key of $\pi$ if and only if the values of $T$ come from a ``southwest'' condition set.

Fix $T \in \mathcal{T}_\lambda$.  Fix $(l, k) \in \lambda$.  For each $j \leq l$, it can be seen that there is exactly one $i \in [1, \zeta_j]$ such that $(l,k) \in P(T; j,i)$.  Now fix some $1 \leq j \leq l-1$.  Let $a(l,k; j) =: a(j)$ be the row index such that $(l-1, k) \in P(T; j, a(j))$.  If $k < \zeta_l$, let $b(l,k; j) =: b(j)$ be the row index such that $(l, k+1) \in P(T; j, b(j))$.  When $k = \zeta_l$, set $b(l,k; j) := \zeta_j + 1$.  It can be seen that the only paths beginning in column $j$ that may reach $(l,k)$ are the paths originating from rows $a(j)$ through row $b(j) - 1$ inclusive.

For $a(j) \leq i \leq b(j)-1$, let $h$ be the largest value less than $l$ such that $(h, m) \in P(T; j, i)$ for some $m$.  Then for such $i$, define $E(l, k; j, i) := T(h,m)$, where $h$ and $m$ depend upon $l,k,j,i$ as above.  By convention, set $a(l) := b(l)-1 := k$ and $E(l,k;l,k) :=k$.  Now refer to $(\lambda, \pi) \in \Lambda_n^+ \times S_n$ and $Y_\lambda(\pi) =: Y$.  Define the set $B_\lambda(T, \pi; l, k) := \bigcap_{j=1}^{l} \hspace{2mm}( \bigcup_{i=a(j)}^{b(j)-1} \hspace{2mm} [ E(l, k; j,i), Y_\lambda(\pi; j,i)  ] \hspace{2mm} )$.  The following result appeared in \cite{Wi1}:

\begin{thm}Given $(\lambda, \pi) \in \Lambda_n^+ \times S_n$, let $T \in \mathcal{T}_\lambda$.  Then $S(T) \leq Y_\lambda(\pi)$ if and only if $T(l, k) \in B_\lambda(T, \pi; l, k)$ for all $(l,k) \in \lambda$.\end{thm}

\begin{proof}Suppose $S(T) \leq Y_\lambda(\pi)$.  Fix $(l, k) \in \lambda$. Let $1 \leq j \leq l-1$.  Let $1 \leq i \leq \zeta_j$ be the unique index such that $(l,k) \in P(T; j,i)$.  Here $S(T;j,i) \leq Y(j,i)$.  The last value before $T(l,k)$ in the EWIS defining $P(T; j,i)$ is $E(l,k;j,i)$.  The last value in this EWIS is $S(T; j, i)$.  So $E(l,k; j,i) \leq T(l,k) \leq S(T;j,i)$.  Hence $T(l,k) \in [ E(l,k; j,i), Y(j,i) ]$.  Note that $a(j) \leq i \leq b(j)-1$.  When $j=l$, we have $\bigcup_{h=a(j)}^{b(j)-1}  [ E(l, k; j,h), Y(j,h)  ]  ) = [k, Y(l,k)]$.  Since $T$ is semistandard, we know $T(l,k) \geq k$.  The definition of $S(T;l,k)$ implies $T(l,k) \leq S(T; l,k)$.  Hence $T(l,k) \in [k, Y(l,k)]$.  Intersecting over $1 \leq j \leq l$, we see $T(l,k) \in B_\lambda(T,\pi;l,k)$.

Now suppose $T(l,k) \in B_\lambda(T, \pi; l, k)$ for all $(l,k) \in \lambda$.  Fix $(j, i) \in \lambda$.  Let $(l,k)$ be the last position in $P(T; j,i)$; here $S(T; j,i) = T(l,k)$.  Since $1 \leq j \leq l$ we have $T(l,k) \in \bigcup_{h = a(j)}^{b(j)-1} $ $[ E(l,k; j,h), Y(j,h) ]$.  However, the value $T(l,k) < E(l,k; j,h)$ for all $h > i$.  (Otherwise $(l,k)$ would be in $P(T; j,h)$ for some $h > i$.)  So $T(l,k) \in \bigcup_{h = a(j)}^i [ E(l,k; j,h) , Y(j,h) ]$.  Since $Y$ is semistandard, we have $Y(j, r) > Y(j, s)$ when $r > s$.  Thus $Y(j, i)$ is an upperbound for $\bigcup_{h = a(j)}^i [ E(l,k; j,h), Y(j,h) ]$.  This implies $S(T; j,i) = T(l,k) \leq Y(j, i)$.
\end{proof}

\section{Left key conditions}

Here we outline results for the left key of a tableau that are analogous to our Section 5 and 6 right key results.  These conditions for a left key to equal or to dominate a given key are expressed in terms of ``southwestern'' values.

Again fix $\lambda \in \Lambda_n^+$, but now fix $\sigma \in S_n^\lambda$.  Form the $\lambda$-key $Y_\lambda(\sigma) =: Y$ of $\sigma$ and let $T \in \mathcal{T}_\lambda$.  We denote the \textit{left key} of $T$ (as in Appendix A.5 of \cite{Ful}) by $L(T)$.  Following Section 5 of \cite{Wi2}, we describe the construction of the \textit{left scanning tableau} $M(T) =: M$.  Let $1 \leq l \leq \lambda_1$.  Remove the columns to the east of the $l^{th}$ column from $T$ (and $\lambda$), and re-use the notation $T^{(l,\zeta_l)}$ to denote this result.  Consider the value $T(l, \zeta_l)$ at the bottom of the $l^{th}$ column of $T^{(l,\zeta_l)}$.  Successively inspecting the values in the columns indexed by $l-1, l-2, ... ,$ find the values beginning with $T^{(l,\zeta_l)}(l,\zeta_l)$ that form the \textit{maximizing weakly decreasing sequence (MWDS)}:  To do so, take the maximum value in the next column to the left that is less than or equal to the most recent entry in the sequence.  Since $T$ is semistandard, one value will be taken from each column to the west.  The locations of these values form the \textit{left scanning path} originating at $T(l, \zeta_l)$;  it is denoted $N(T;l,\zeta_l)$.  The value of $M(l,\zeta_l)$ is defined to be the last value in this path; it is in the first column of $T$.  Remove the locations in and beneath $N(T;l,\zeta_l)$ from $\lambda$ and the corresponding values from $T$.  Since that path was ``southwesterly'', this will produce a shape and a tableau denoted $T^{(l,\zeta_l-1)}$.  Repeat this process to successively find and remove $N(T;l,k)$ and the values below it from $T^{(l,k)}$ to produce $T^{(l,k-1)}$ for $k = \zeta_l-1, \zeta_l - 2, ... , 1$.  Here $M(l,k)$ is defined at each stage to be the last value in $N(T;l,k)$.  Once this has been done for every $1 \leq l \leq \lambda_1$, the left scanning tableau $M(T)$ has been constructed.  According to Section 5 of \cite{Wi2}, we have $L(T) = M(T)$.

Fix some $(l,k) \in \lambda$.  First suppose $l \geq 2$.  Define \textbf{$V^{(l,k)} =: V$} to be the tableau produced by removing the rightmost remaining column from $T^{(l,k)}$ (and $\lambda$).  Let $q$ be maximal such that $V(l-1, q) \leq T(l,k+1) - 1$.  Then find $N(V;l-1,q), N(V;l-1, q-1), ...$.  (Do not remove these paths as they are formed.)  Note that if $k \leq h < i \leq q$, then $N(V;l-1,h)$ stays weakly above $N(V;l-1,i)$.  Let $g_q, g_{q-1},...$ be the ending values in the first column of $V$ (and hence $T$) of these paths.  Note that the ordering of the paths implies $g_q \geq g_{q-1} \geq ...$.  Let $p$ be minimal such that $g_p \geq Y(l,k)$:  Then $N(V;l-1,p)$ is the last path that needs to be considered, where $p \geq k$.  If no such $p$ exists, define the set $F_\lambda(T,\sigma;l,k) := \emptyset$.  Otherwise define $F_\lambda(T,\sigma;l,k) := [ T(l-1,p), T(l,k+1) - 1]$.  When $l = 1$, define $F_\lambda(T, \sigma; l, k) := [Y(l,k), T(l,k+1) - 1]$.

\begin{thm}Given $(\lambda, \sigma) \in \Lambda_n^+ \rtimes S_n^\lambda$, let $T \in \mathcal{T}_\lambda$.  Then $M(T) \geq Y_\lambda(\sigma)$ if and only if $T(l,k) \in F_\lambda(T,\sigma;l,k)$ for all $(l,k) \in \lambda$.\end{thm}

\begin{proof}  Let $(l,k) \in \lambda$.  The case $l=1$ is obvious.  Suppose $l \geq 2$.  By semistandardness $T(l,k) \leq T(l,k+1) - 1$.  So when forming the MWDS for $M(l,k)$ we need consider only values within the locations $(l-1,q), (l-1, q-1), ... , (l-1,k)$ where $q$ is maximal such that $T(l-1,q) \leq T(l,k+1) - 1$ and such that $(l-1,q)$ was not in a left scanning path for a location $(l,h)$ with $h > k$. Refer to the definition of $F_\lambda(T,\sigma;l,k)$ for the entities $q, p, g_q, g_{q-1}, ... , g_p$.

First suppose that $M \geq Y$.  For the sake of contradiction, suppose $T(l,k) < T(l-1,p)$.  Let $p > h \geq k$ be such that the MWDS from $(l,k)$ passes through $(l-1,h)$.  By the minimality of $p$ we have $g_h < Y(l,k)$.  But since $g_h = M(l,k)$, this would yield the contradiction $M(l,k) < Y(l,k)$.  Hence $T(l,k) \geq T(l-1,p)$.  So $T(l,k) \in F_\lambda(T,\sigma;l,k)$.

Next suppose that $T(j,i) \in F_\lambda(T,\sigma;j,i)$ for all $(j,i) \in \lambda$.  Since $F_\lambda(T,\sigma;l,k)$ is non-empty, we have $T(l-1,p) \leq T(l,k) \leq T(l,k+1) - 1$.  Let $q \geq h \geq p$ be such that the MWDS from $(l,k)$ passes through $(l-1,h)$.  Then $M(l,k) = g_h \geq g_p \geq Y(l,k)$. \end{proof}

Now we constrain $T$ so that $M(T) = Y_\lambda(\sigma)$:  Let $q, p, g_q, ... , g_p$ be as above.  Let $a$ be minimal and $b$ maximal such that $g_a = Y(l,k) = g_b$.  If no such $a,b$ exist, define the set $G_\lambda(T,\sigma;l,k) := \emptyset$.  Otherwise define $G_\lambda(T,\sigma;l,k) := [ T(l-1, a), min \{ T(l-1,b+1) - 1, T(l, k+1) - 1 \} ]$.  (When $l = 1$, define $G_\lambda(T,\sigma;l,k) := \{ Y(l,k) \}$.)  The proof of the next result is similar to that of Theorem 9.1:

\begin{thm}Given $(\lambda, \sigma) \in \Lambda_n^+ \rtimes S_n^\lambda$, let $T \in \mathcal{T}_\lambda$.  Then $M(T) = Y_\lambda(\sigma)$ if and only if $T(l,k) \in G_\lambda(T,\sigma;l,k)$ for all $(l,k) \in \lambda$.\end{thm}

\section{Conclusions}

Using the sets that were developed using the scanning viewpoints in Sections 5-9, the following applications to the original right or left key viewpoint and to polynomials may be stated for a fixed pair of choices $(\lambda, \pi) \in \Lambda_n^+ \rtimes S_n^\lambda$:

\begin{thm}Let $T$ be a semistandard tableau of shape $\lambda$.  The following are equivalent:

\noindent (i)  $T$ is a Demazure tableau for $\pi$ (that is, $R(T) \leq Y_\lambda(\pi)$),

\noindent (ii)  $T(l,k) \in A_\lambda(T, \pi; l, k)$ for all $(l,k) \in \lambda$, and

\noindent (iii)  $T(l,k) \in B_\lambda(T,\pi;l,k)$ for all $(l,k) \in \lambda$. \end{thm}

\begin{cor}The Demazure character $d_\lambda(\pi;x)$ is the sum of $x^T$ over all $T \in \mathcal{T}_\lambda$ such that

\noindent (i)  $T(l,k) \in A_\lambda(T,\pi;l,k)$ for all $(l,k) \in \lambda$, or

\noindent (ii)  $T(l,k) \in B_\lambda(T,\pi;l,k)$ for all $(l,k) \in \lambda$.\end{cor}

\begin{thm}A semistandard tableau $T$ of shape $\lambda$ is an exact Demazure tableau at $\pi$ (that is, $R(T) = Y_\lambda(\pi)$) if and only if $T(l,k) \in C_\lambda(T,\pi;l,k)$ for all $(l,k) \in \lambda$.\end{thm}

\begin{cor}The atom $c_\lambda(\pi;x)$ is the sum of $x^T$ over all $T \in \mathcal{T}_\lambda$ such that $T(l,k) \in C_\lambda(T,\pi;l,k)$ for all $(l,k) \in \lambda$.\end{cor}

\begin{thm}Procedure 7.1 produces all semistandard tableaux whose right keys are the $\lambda$-key of $\pi$, that is $\mathcal{V}^{[1]} = \{ T \hspace{2mm} | \hspace{2mm} R(T) = Y_\lambda(\pi) \}$.\end{thm}

\begin{cor}The atom $c_\lambda(\pi;x)$ is the sum of $x^T$ over all $T \in \mathcal{V}^{[1]}$.\end{cor}

\noindent Now also fix some $\sigma \in S_n^\lambda$:

\begin{thm}A semistandard tableau $T$ has $Y_\lambda(\sigma) \leq L(T)$ if and only if $T(l,k) \in F_\lambda(T,\sigma;l,k)$ for all $(l,k) \in \lambda$, and it has $Y_\lambda(\sigma) = L(T)$ if and only if $T(l,k) \in G_\lambda(T,\sigma;l,k)$ for all $(l,k) \in \lambda$.\end{thm}

Has the polynomial $\sum x^T$, sum over $T$ such that $Y_\lambda(\sigma) \leq L(T)$ and $R(T) \leq Y_\lambda(\pi)$, been considered?  Here $T(l,k) \in F_\lambda(T,\sigma;l,k) \bigcap B_\lambda(T,\pi;l,k)$, an intersection of two southwestern condition sets.  For this polynomial to be non-zero, one must have $\sigma \leq \pi$ in the Bruhat order on $S_n^\lambda$, since $L(T) \leq T \leq R(T)$ would imply $Y_\lambda(\sigma) \leq Y_\lambda(\pi)$.  Demazure introduced Demazure polynomials while studying the desingularization of Schubert varieties.  Kazhdan-Lusztig polynomials are also indexed by intervals in these Bruhat orders and are related to the structure of singularities of Schubert varieties.

All of our results are ``stable'' as  $n  \rightarrow \infty$ for $\pi \in S_\infty$ (as defined in \cite{RS1}).   So the polynomial results hold in infinitely many variables  $x_1, x_2, ... $ for unbounded tableaux.

\section*{Appendix:  Interface with representation theory}

The ingredients needed to define Demazure modules of semisimple Lie algebras and their characters are in \cite{Hum}:  Given a complex semisimple Lie algebra $L$, choose a Cartan subalgebra $H$ and a Borel subalgebra $B \supseteq H$.  These choices determine the rank  $n := dim(H)$ of $L$, and then (for $1 \leq i \leq n$) the simple roots $\alpha_i \in H^*$, the simple reflections  $s_i$ of $H^*$, and the fundamental weights $\omega_i$.  The simple reflections generate the Weyl group $W$ and the fundamental weights generate the weight lattice $\Lambda$, which contains the set of dominant weights $\Lambda^+$.  Fix $\lambda \in \Lambda^+$.  Let $V_\lambda$ be a finite dimensional irreducible $L$-module with highest weight $\lambda$.  Let  $w \in W$.  Let $v_{w\lambda} \neq 0$ be a weight vector of weight $w\lambda$.  The Demazure module $D_\lambda(w)$ is the $B$-submodule $\mathcal{U}(B).v_{w\lambda}$ of $V_\lambda$, where $\mathcal{U}(B)$ is the universal enveloping algebra of $B$.  The lowest weight of this module is $w\lambda$.  When $w$ is the longest element  $w_0$  of  $W$,  one has  $D_\lambda(w_0) = V_\lambda$.  For each  $\mu \in \Lambda$  there is a formal exponential  $e^\mu$.  Given  $\mu \in \Lambda$,  let  $m_\lambda(w,\mu)$  be the dimension of the  $H$-weight space of $D_\lambda(w)$  of weight  $\mu$.  The formal character  $char_\lambda(w)$  of  $D_\lambda(w)$  is  $\sum m_\lambda(w,\mu)e^\mu$,  where the sum runs over $\mu \in \Lambda$.  The formal character of the  $L$-module  $V_\lambda$  is $char_\lambda(w_0)$.  For some  $k \geq 0$,  let  $s_{i_k}...s_{i_2}s_{i_1}$  be a reduced decomposition for  $w$.  Taking $\partial_i(e^\mu)  :=  (e^\mu - e^{s_i\mu - \alpha_i})/(1 - e^{-\alpha_i})$  for  $\mu \in \Lambda$, the Demazure character formula (Equation 8.2.9.4 of \cite{Kum}) is $char_\lambda(w)  =  \partial_{i_k}...\partial_{i_2}\partial_{i_1}.e^\lambda$.  To precisely index the Demazure submodules of  $V_\lambda$,  first set  $J := J_\lambda := \{ i \in [n]:  s_i.\lambda = \lambda \}$.  Here  Stab$_W  (\lambda) = \langle s_i :  i \in J \rangle =: W_J $.  (If  $\lambda = \sum_{1 \leq i \leq n} a_i\omega_i$  for some  $a_i \in \mathbb{N}$,  then  $J = \{ i \in [n]:  a_i = 0 \}$ .)  There is one distinct Demazure module for each coset  $wW_J$  in the set of cosets  $W^J := W/W_J$.  Each such coset has a unique minimal length representative in $W$;  let $W^\lambda$ denote the set of these representatives.

Now take  $L$  to be the simple Lie algebra  $sl_n(\mathbb{C})$;  it has rank  $n-1$.  Here $W \cong S_n$, the symmetric group.  Choose  $H$  to be the subspace of diagonal matrices and  $B$  to be the subalgebra of trace free upper triangular matrices.  For  $1 \leq i \leq n$,  define  $\phi_i \in H^*$  to be the linear function that extracts the coefficient of the elementary matrix  $E_{ii}$  for each element of  $H$.  Note that  $\phi_1 + \phi_2 + ... + \phi_n = 0$  on  $H$.  Let $\textbf{E}$   denote the real span of  $\phi_1, \phi_2, ... , \phi_n$.  For $1 \leq i \leq n-1$, we have $\alpha_i = \phi_i - \phi_{i+1}$  and  $\omega_i = \phi_1 + \phi_2 + ... + \phi_i$  on  $H$.

In this paper we avoid using an action from the right (or mentioning  $w^{-1}$)  by using \textit{two} combinatorial models for the action of  $W$ from the left.  For the first model,  note that  $s_i.\phi_i = \phi_{i+1}$,  $s_i.\phi_{i+1} = \phi_i$,  and  $s_i.\phi_j = \phi_j$  when $j \notin \{ i,i+1 \} $.  Set  $x_i := e^{\phi_i}$.  Note that  $x_1x_2 \cdots x_n = 1$.  There is an induced action of  $W$  on the set of formal exponentials:  Here  $s_i.x_i = x_{i+1}$,  $s_i.x_{i+1} = x_i$,  and  $s_i.x_j = x_j$  when  $j \notin \{ i,i+1 \}$.  This is the same as the second action of the $s_i$ in Section 2, on polynomials.  Here we say that  $W$  is ``acting by value'' on the subscripts.  This induces the first action of the $s_i$ in Section 2, on permutations.  When using this model for $W$,  we often refer to the permutation $\pi := (\pi_1, ... , \pi_n) := \pi_w := w.(n)$.

Each $\mu \in \Lambda$  may be uniquely represented in the form  $\sum_{1 \leq i \leq n-1} b_i \omega_i$  for some $b_i \in \mathbb{Z}$.  Fix some $\lambda \in \Lambda^+$ and write  $\lambda =: \sum_{1 \leq i \leq n-1} a_i\omega_i$  for some  $a_i \in \mathbb{N}$.  Here the symbol  $\lambda$  is being used in the traditional Lie-theoretic manner. Transitioning to the traditional combinatorial usage of  $\lambda$,  set  $\lambda_i := \sum_{i \leq j \leq n-1} a_j$  for  $1 \leq i \leq n-1$  and  $\lambda_n := 0$.  This is the  $i^{th}$ coefficient of  $\lambda$ with respect to the  $\{ \phi_i \}$  spanning set for $\textbf{E}$   when  $\lambda_n$  is required to vanish.  Since  $\lambda_1 \geq \lambda_2 \geq ... \geq  \lambda_{n-1} \geq  \lambda_n = 0$,  this produces an  $n$-partition which will also be denoted  $\lambda$.  This partition  $\lambda$  is strict if and only if the weight  $\lambda$  is strongly dominant.  For the second combinatorial model of the action of  $W$,  note that for $1 \leq i \leq n-1$  one has the reflection action  $s_i.(\lambda_1,...,\lambda_i,\lambda_{i+1},...,\lambda_n)^T = (\lambda_1,...,\lambda_{i+1},\lambda_i,...,\lambda_n)^T$  on column vectors of coefficients with respect to  $\{ \phi_i \}$.  Here we say that  $W$  ``acts by position''.  When using this model for  $W$,  we often depict  $w \in W$  with a reduced decomposition  $s_{i_k}...s_{i_2}s_{i_1}$  for some  $k \geq 0$.  The orbit  $W\lambda$  consists of all of the ``shuffles'' of the multiset of  $n$ integers  $\{ \lambda_i \}_{1 \leq i \leq n}$;  these are called ``compositions [of the integer $|\lambda|$ ]'' in \cite{RS1}.  Note that  $J = \{ i \in [n-1]:  \lambda_i = \lambda_{i+1} \}$,  and so  $J$  can be used to describe the presence of multiplicities amongst the  $\lambda_i$.  These shuffles correspond exactly to the elements of $W^\lambda$.  The set of column lengths that may possibly occur in the Young diagram of  $\lambda$ is  $[n-1]$.  Since  $J$  is the set of ``missing'' column lengths,  comparing to Section 2 we have  $J = [n-1] - Q_\lambda$.

The Weyl character formula for the coordinatization of $char_\lambda(w_0;x)$  for $sl_n(\mathbb{C})$ is the bialternant definition of the Schur function  $s_\lambda(x)$.  So $s_\lambda(x) = \sum_{T \in \mathcal{T}_\lambda} x^T$ implies the dimension $m_\lambda(w_0,\mu)$  is the number of tableaux  $T$  such that $c_i$ is the  $i^{th}$ coefficient of  $\mu$  with respect to the  $\{ \phi_i \}$  spanning set when the coefficients of $\mu$ are required to sum to  $| \lambda |$.  Let  $T^+$  be the tableau of shape  $\lambda$  that has  $T^+(j,i) = i$  for $1 \leq i \leq n$  and  $1 \leq j \leq \lambda_i$.  Here $x^{T^+}$ is the coordinatization  $x_1^{\lambda_1} \cdots x_{n-1}^{\lambda_{n-1}}x_n^0$  of $e^\lambda$.

Consider a composition  $\alpha \in W\lambda$.  Let  $w \in W^\lambda$  be of length  $k \geq 0$  such that  $w.\lambda = \alpha$.  Let  $s_{i_k}...s_{i_2}s_{i_1}$  be a reduced decomposition for  $w$, and find $\pi = \pi_w$.  To relate to the ``right action'' of \cite{RS1},  note that  $\lambda_i = \alpha_{\pi_i}$  for  $1 \leq i \leq n$.  Now identify each value $1 \leq i \leq n$  of an  $n$-semistandard tableau  $T$  of shape  $\lambda$  with the formal exponential  $x_i$.  Define  $w.T^+$  to be the result of replacing each value  $i$  by  $\pi_i$  and resorting the values within each column so that they increase from north to south. Clearly  $w.T^+ = Y_\lambda(\pi)$.  The combinatorial weight $x^{w.T^+}$ of this tableau is the coordinatization of $e^{w\lambda}$.  In \cite{RS1} the ``key'' $key(\alpha)$ of the composition $\alpha$ is defined to be the tableau whose westernmost  $\alpha_j$  columns contain the value $j$ for $j \geq 1$.  Taking  $j := \pi_i$  for a given  $i \geq 1$,  one sees that the tableau  $w.T^+$ satisfies that definition.  Hence  $Y_\lambda(\pi) = key(\alpha)$.

It is not hard to see that the coordinatization of the Demazure character formula above is our definition (when $\lambda_n = 0$) of the Demazure polynomial $d_\lambda(\pi;x)$ in Section 2.  The dimension $m_\lambda(w, \mu)$ is the number of Demazure tableaux $T$ such that $c_i$ is the $i^{th}$ coefficient of $\mu$.

Since the reductive Lie algebra  $gl_n(\mathbb{C})$  is not semisimple,  its Demazure modules are rarely considered in geometric or algebraic papers.   However, its coordinatized characters have some aesthetic advantages over those for $sl_n(\mathbb{C})$.   Since the scalar matrices are in the center of  $gl_n(\mathbb{C})$,  the familiar constructions (such as with tensors or global sections of line bundles on  $SL_n(\mathbb{C})/B$ )  of an  $sl_n(\mathbb{C})$  module  $V_\lambda$  may be readily extended to $gl_n(\mathbb{C})$.   Once the relation  $\phi_1 + \phi_2 + ... + \phi_n = 0$  is no longer present,  the finite dimensional irreducible polynomial characters of  $gl_n(\mathbb{C})$  are indexed by the Young diagrams for  $n$-partitions:   for each column of length  $n$  in the shape of a  $\lambda \in \Lambda_n^+$,  the formal character has a factor of $x_1x_2 \cdots x_n$.   Every Schur function  $s_\lambda(x)$  for  $\lambda \in \Lambda_n^+$  now arises as a formal character for  $gl_n(\mathbb{C})$.  Since the underlying vector spaces for the modules are unchanged,  their structure with respect to  $W \cong S_n$  remain the same. Extend $B$  to the subalgebra  $B^\prime$  of all upper triangular matrices in  $gl_n(\mathbb{C})$,  consider highest weight vectors  $v$  for all  $n$-partitions  $\lambda \in \Lambda_n^+$,  and construct  $\mathcal{U}(B^\prime).wv$  for any  $w \in W$.   Now that the condition  $\lambda_n = 0$  has been removed,  \textit{all} of the Demazure polynomials  $d_\lambda(\pi;x)$  considered in this paper arise as the formal characters for such ``polynomial'' Demazure modules of  $gl_n(\mathbb{C})$  as the initial monomial  $x_1^{\lambda_1}x_2^{\lambda_2} \cdots x_n^{\lambda_n}$  ranges through all  $\lambda \in \Lambda_n^+$.   Columns of length  $n$  in a tableau  $T \in \mathcal{T}_\lambda$  must contain the values  $1, 2, ..., n$.   It can be seen that such columns are combinatorially inert in this paper.   Hence our Theorem 10.1 may be applied to the Demazure characters for  $sl_n(\mathbb{C})$  by requiring  $\lambda$  to be an  $(n-1)$-partition and invoking the relation  $x_1x_2 \cdots x_n = 1$.

In the semisimple  $L$  and coordinatized  $sl_n(\mathbb{C})$  discussions above,  to avoid redundant considerations of cases we required  $w \in W^\lambda$.  A permutation $\pi \in S_n$  corresponds to a  $w \in W^\lambda$  if and only if  $\pi \in S_n^\lambda$.   The criteria for having some redundant $w$ is the same for the  $gl_n(\mathbb{C})$ case (when  $\lambda$  is any  $n$-partition) as for the  $sl_n(\mathbb{C})$  case (when  $\lambda_n = 0$):  whether any of the parts of  $\lambda$  are repeated.  Let $\pi, \pi^\prime \in S_n$.  By Proposition 2.4.4 of [BB], any $w \in W$ can be uniquely factored as $w = w_2w_1$ such that $w_1 \in W_J$ and $w_2 \in W^\lambda$.  Creating $Y_\lambda(\pi)$ from $\pi$ is essentially projecting $\pi$ to $S_n^\lambda$:  This proposition can be used to show that $Y_\lambda(\pi) = Y_\lambda(\pi^\prime)$ if and only if $w_2 = w_2^\prime$.  It can also be used to show that $c_\lambda(\pi;x) = 0$ if and only if $\pi \notin S_n^\lambda$.  For a fixed  $\lambda$,  let $\pi, \pi^\prime \in S_n^\lambda$  correspond to  $w,w^\prime \in W^\lambda$.   Then  $w^\prime \leq w$  in the Bruhat order on  $W^\lambda$ if and only if  $Y_\lambda(\pi^\prime) \leq Y_\lambda(\pi)$  by Theorem 2.6.3 of \cite{BB}.   This gives the restatement  $d_\lambda(w;x) = \sum c_\lambda(w;x)$,  sum over all  $w^\prime \in W^\lambda$  such that  $w^\prime \leq w$  in the Bruhat order on  $W^\lambda$.  As $\lambda$ runs through $\Lambda_n^+$, the union of the orbits $\{ w.\lambda \hspace{1mm} | \hspace{1mm} w \in W^\lambda \}$ is $\mathbb{N}^n$.  Here the correspondence $w \mapsto \pi_w$ can be used to describe a bijection from $\Lambda_n^+ \rtimes S_n^\lambda$ to $\mathbb{N}^n$.

The notions of right and left keys of a semistandard tableau are related to the lifting criterion of Lakshmibai, Musili, and Seshadri for standard monomials in Type A.  See Section 3 of \cite{RS2} and Section 12.8 of \cite{LB}.  The notions of right and left keys have been generalized to analogous constructions for all semisimple Lie algebras and Kac-Moody algebras:  The right (left) keys are the initial (final) directions of Littelman's Lakshmibai-Seshadri paths.  See Lenart's Remark 5.3 in \cite{Le2} for the details in the general Lenart-Postnikov alcove path model or Proposition 3.4.3 of \cite{Fer} for Type A.  The left key plays a role for ``opposite Demazure'' modules that is analogous to the role played by the right key for Demazure modules. The generating procedure on p. 281 of \cite{Le1} produces the ``Demazure crystal graph'' whose vertices are the tableaux in $\mathcal{D}_\lambda(\pi).$

\vspace{1pc}\noindent \textbf{Acknowledgments.}  The first author thanks Soichi Okada and Masao Ishikawa for organizing the August 2012 RIMS conference on Young tableaux.  Alain Lascoux's remarks there inspired the results of this paper beyond Theorem 8.1, which had appeared in \cite{Wi1}.  He kindly sent us several helpful messages in the months following the conference, as we were writing this paper.  We thank Shrawan Kumar, David Lax, and Joseph Seaborn for various remarks.  We also thank the two referees for their helpful thorough reports, one of which supplied the missing restriction ``$\pi \in S_n^\lambda$'' for Theorem 3.2 and Corollaries 10.4 and 10.6.

\newpage
\begin{spacing}{1.125}

\begin{center}

Minor Improvements for

\textbf{``Semistandard Tableaux for Demazure Characters \\ (Key Polynomials) and Their Atoms''}

by Robert A. Proctor and Matthew J. Willis

July 5, 2017

\end{center}

\noindent\textbf{(1)  Specification of scanning tableau  $\mathbf{\emph{S(T)}}$}

\vspace{1pc}\noindent In some places in this paragraph on p. 8 the notation was not as precise as it should have been.  Below red ink is used to indicate six insertions which make the notation more precise and the presentation clearer.

\vspace{1pc}Let $1 \leq l \leq \lambda_1$.  Create $T^{(l, \zeta_l)}$ \color{red} and $\lambda^{(l, \zeta_l)}$ \color{black} from $T$ by removing the first $l-1$ columns from $T$ and $\lambda$, but retain the column indexing.  We compute the values in the $l^{th}$ column of $S(T)$ from $(l, \zeta_l)$ upwards:  Consider the column bottom values $T^{(l,\zeta_l)}(h, \zeta_h)$ for $l \leq h \leq \lambda_1$ as a sequence, and find its EWIS.  The sequence of locations that contain the values of this EWIS is the \emph{scanning path} for this location; it is denoted $P(T;l,\zeta_l)$.  The first member of $P(T;l,\zeta_l)$ is the location $(l,\zeta_l)$.  Begin to create \color{red} the $l^{th}$ column of \color{black} $S(T)$ by defining the value $S(T;l,\zeta_l)$ to be the last value in this EWIS.  Next remove the boxes in $P(T; l, \zeta_l)$ from $\lambda\color{red}^{(l, \zeta_l)}$ \color{black} and their values from $T\color{red}^{(l, \zeta_l)}$ \color{black} to form what can be seen to be a smaller shape \color{red} $\lambda^{(l, \zeta_l-1)}$ \color{black} and a \textit{remnant} tableau $T^{(l,\zeta_l-1)}$.  Since $T^{(l,\zeta_l-1)}$ is semistandard, we may apply $S(\cdot)$ to it.    As $k$ decrements from $\zeta_l - 1$ to 1, continue to perform this process using the bottom values in the $l^{th}$ through $\lambda_1^{th}$ columns of the diminishing $T^{(l,k)}$ to produce the other $\zeta_l - 1$ scanning paths that originate in the $l^{th}$ column.  For such $k$, the path constructed with the selected column bottoms of $T^{(l,k)}$ is denoted $P(T;l,k)$, and $S(T;l,k)$ is defined to be the value in its final location.  Note that $S(T;l,k)$ is the largest of the column bottom values in $T^{(l,k)}$, i.e. the largest value in $T^{(l,k)}$.  Apply this process to all of the columns of $T$ to obtain the scanning value $S(T; l,k)$ for every $(l,k) \in \lambda$.  Define $U^{(l,k)}$ to be the tableau produced by removing the leftmost remaining column from $T^{(l,k)}$ and $\lambda\color{red}^{(l,k)}$.  \color{black} To summarize, with the second equality giving the form used in Sections 5 - 7:

\vspace{2pc}\noindent\textbf{(2)  Added comment to the Appendix}

\vspace{1pc}\noindent  The following observation could have been made between the $7^{th}$ and $8^{th}$ sentences in the paragraph that straddles pp. 15-16:

\vspace{1pc}\noindent ``The $n$-permutations in $S_n^\lambda$ can be used to depict the minimal length coset representatives in $W^\lambda$ here.''

\end{spacing}

\end{document}